\documentclass[]{article}
\usepackage{graphicx}
\usepackage{caption}
\usepackage{amsfonts}
\usepackage{amsmath}
\usepackage{amssymb}
\usepackage{fancyhdr}
\usepackage{titlesec}
\usepackage{indentfirst}
\usepackage{booktabs}
\usepackage{verbatim}
\usepackage{color}
\usepackage{tcolorbox}
\usepackage{mathtools}
\usepackage{bm}
\usepackage{amsthm}
\usepackage{wasysym}
\usepackage{empheq}
\graphicspath{{figures/}} 
\usepackage[font=small,labelfont=bf]{caption} 
\usepackage{chngcntr}

\newcommand{\rom}[1]{\uppercase\expandafter{\romannumeral #1\relax}}
\usepackage[page,header]{appendix}
\usepackage{subfig}
\usepackage{cases}
\usepackage{titletoc}

\numberwithin{equation}{section}
\newtheorem{theorem}{Theorem}[section]
\newtheorem{lemma}[theorem]{Lemma}

\newtheorem{proposition}[theorem]{Proposition}

\newtheorem{remark}[theorem]{Remark}

\newtheorem{assumption}[theorem]{Assumption}

\usepackage{tikz}

\def\a{\tau}
\def\k{\kappa}
\def\b{\gamma}
\def\d{\delta}
\def\e{\epsilon}
\def\g{\gamma}
\def\lam{\lambda}
\def\o{\omega}
\def\p{\phi}

\def\th{\theta}
\def\v{\varepsilon}
\def\L{\Lambda}

\def\p{\rho}
\def\S{\Sigma}

\def\E{\mathbb E}

\def\R{\mathbb R}

\def\Mx{\mathcal P}

\def\l{\left}
\def\r{\right}
\def\la{\left\langle}
\def\ra{\right\rangle}
\def\ll{\left\lVert}
\def\rl{\right\rVert}
\def\lv{\left\lvert}
\def\rv{\right\rvert}
\def\({\left(}
\def\){\right)}
\def\[{\left[}
\def\]{\right]}

\def\pt{\partial}
\def\nb{\nabla}

\def\ds{\displaystyle}

\def\qd{\quad}

\def\h{\hat}
\def\t{\tilde}

\def\p{\psi}
\def\Mx{M_x}
\def\Mv{M_v}

\def\st{*}
\def\T{T}
\def\L{L}

\def\K{R}
\def\cc{\e_0}
\def\bb{C_2}
\def\w{\o_0}

\begin{document}

\title{A Sharp Convergence Rate for the Asynchronous Stochastic Gradient Descent}

\author{ 
Yuhua Zhu\footnote{Department of Mathematics, Stanford University. (yuhuazhu@stanford.edu)}\quad
Lexing Ying\footnote{Department of Mathematics, Stanford University. (lexing@stanford.edu)}
}

\date{}
\maketitle

\newcommand\todo[1]{{{\color{red}To do: #1}}}


\begin{abstract}
{We give a sharp convergence rate for the asynchronous stochastic gradient descent (ASGD) algorithms
when the loss function is a perturbed quadratic function based on the stochastic modified equations
introduced in [An et al. Stochastic modified equations for the asynchronous stochastic gradient
  descent, arXiv:1805.08244]. We prove that when the number of local workers is larger than the
expected staleness, then ASGD is more efficient than stochastic
gradient descent. Our theoretical result also suggests that longer delays result in slower convergence rate. Besides, the learning rate cannot be smaller than a threshold inversely proportional to the expected staleness. }
\end{abstract}




\section{Introduction}
\label{sec:Intro}
Thanks to the availability of large datasets and modern computing resources, optimization-based
machine learning has achieved state-of-the-art results in many applications of artificial
intelligence. As the datasets continue to increase, distributed optimization algorithm have received
more attention for solving large scale machine learning problems. Parallel {\em stochastic gradient
  descent} (SGD) is arguably the most popular one among them.


Based on the interaction between different working nodes, there are two types of parallel SGD
algorithms: synchronous v.s. asynchronous SGD (SSGD v.s. ASGD). Both methods compute the gradient of
the loss function for a given mini-batch on local workers. In SSGD, the local workers pause the
training process until the gradients from all local works have been added into the {\em shared}
parameter variable. While in ASGD, the local workers interact with the shared parameter
independently without any synchronization, i.e., each local worker continues to compute the next
gradient right after their own gradients have been added to the shared parameter. Therefore, ASGD is
presumably more efficient than SSGD since the overall training speed is not affected by the slow
local workers. On the other hand, ASGD can potentially suffer from the problem of delayed gradients,
i.e., the gradients that a local worker sends to the shared parameter are often computed with
respect to the parameter of an older version of the model. Therefore, extra stochasticity is
introduced in ASGD due to this delay. An interesting mathematical problem is how the delayed
gradient affects the training process.


There have been a few papers in the literature that analyze the convergence rate of ASGD, but most
of them from an optimization perspective. For example,
\cite{recht2011hogwild,duchi2013estimation,mania2015perturbed} proved that ASGD can achieve a nearly
optimal rate of convergence when the optimization problem is sparse; \cite{mitliagkas2016asynchrony}
studies the relationship between ASGD and momentum SGD.


This note follows a perspective based on partial differential equation (PDE) and stochastic
differential equation (SDE). In \cite{li2017stochastic}, Li et al first introduced the stochastic
modified equations (SME) for modeling the dynamics of SGD in a continuous time
approximation. Following this work, there have been quite a few papers along this direction
\cite{jastrzkebski2017three,mandt2017stochastic,chaudhari2018stochastic,hu2019diffusion} for SGD.
Recently, in \cite{AnLuYing2018}, An et al applied the SME approach to the study of ASGD.


This note studies the convergence rate of the time-dependent probability distribution function (PDF)
of ASGD to its steady state distribution by using PDE techniques of the stochastic modified
equation. The main focus is on the case where the loss function is a perturbed quadratic
function. There are mainly two difficulties in this analysis. The first one is that asynchrony
results in a degenerate diffusion operator in the corresponding PDE. No trivial analysis is able to
give an exponential decay rate for the convergence to the steady state. Thanks to the association of
a degenerate diffusion operator and a conservative operator, the decay rate can be recovered through
\textquotedblleft hypocoercivity" \cite{villani2009hypocoercivity}. The key here is to construct a
Lyapunov functional to prove the exponential decay of this functional. The second difficulty is to
obtain a sharp convergence rate. Such a sharp rate is important to understand the influence of the
asynchrony quantitatively. Though our analysis is based on the framework introduced by
\cite{arnold2014sharp}, the current case is more complicated because one has to bound extra terms
introduced by the perturbed loss function around the quadratic function. The main observations of
this paper are the following:
\begin{itemize}
\item We give a sharp convergence rate for ASGD when the loss function is a perturbed quadratic function.
\item For a fixed learning rate, longer delays result in slower convergence rate.
\item The learning rate should not be smaller than a threshold and the threshold is inversely proportional to the staleness rate. See Remark \ref{rmk: staleness conv} for details. 
\item When the number of local workers is larger than the expectated staleness, then
  ASGD is more efficient than SGD. See Remark \ref{rmk: ASGD SGD} for details. 
\end{itemize}


The rest of the paper is organized as follows. Section \ref{sec: model} summarizes the results of
\cite{AnLuYing2018} and derives the PDE for the probability density function (PDF) of ASGD based on
these results. Section \ref{sec: main results} presents the main results, while the proof of the
main theorem is given in Section \ref{sec: proof of main thm}.

\section{PDE for the probability density function of ASGD}
\label{sec: model}
We consider the minimization problem,
\begin{equation}
  \label{eq: main op}
  \min_{\th\in\R^d} f(\th) := \frac{1}{n}\sum_{i=1}^nf_i(\th).
\end{equation}
where $\th$ represents the model parameters, $f_i(\th)$ denotes the loss function at the $i$-th
training sample and $n$ is the size of the training sample set.  In the asynchronous stochastic
gradient descent (ASGD) algorithm, the parameter $\th$ is updated with
\begin{equation*}
  \th_{k+1} = \th_k - \eta\nb_\th f_{\g_k}(\th_{k-\tau_k}),
\end{equation*}
where $\g_k$ is i.i.d. uniform random variable from $\{1, 2, \cdots, n\}$ and $\th_{k-\tau_k}$ is
the delayed read of the parameter $\th$ used to update $\th_{k+1}$ with a random staleness $\tau_k$.

{In \cite{AnLuYing2018}, An et al derived the modified stochastic differential equation for the
algorithm, under the assumption that $\tau_k$ follows the geometric distribution, i.e., $\tau_k = l$
with probability $(1-\k)\k^l$ for $\k\in (0,1)$. We call $\k$ the {\it staleness rate}. Note that if $\k$ is larger, then there are a longer delay. Besides, the expectation of the random staleness $\tau_k$ is $\frac{1}{1-k}$, so we call $\frac{1}{1-k}$ the {\it expected staleness}. }By introducing
\begin{equation*}
  y_k = -\sqrt{\frac{\eta}{1-\k}} \E_{\tau_k}\nb f(\th_{\tau_k}), 
\end{equation*}
when the learning rate $\eta$ is small, $(\th_k, y_k)$ can be approximated by time discretizations
of a continuous time stochastic process $\l(\Theta_{k\d_t}, Y_{k\d_t}\r)$ for
$\d_t=\sqrt{\eta(1-\k)}$. $\l(\Theta_t, Y_t\r)$ satisfies the stochastic differential equation
(SDE)
\begin{equation}
  \begin{aligned}
    \label{eq: main SME}
    &d\Theta_t = Y_tdt + \a dB_t,\\
    &dY_t = -\nb f(\Theta_t)dt - \b Y_tdt,
  \end{aligned}
\end{equation}
where $\a = \eta^{3/4}/(1-\k)^{1/4}\S$, $\b = \sqrt{((1-\k)/\eta)}$, and $\S$ is the covariance
matrix conditioned on $\tau_k$, which is assumed to be a constant.

We first formally derive the partial differential equation for the probability density function
$\p(t,\th,y)$ of $(\Theta_t, Y_t)$. For any compactly supported $C^\infty$ function
$\phi(\th_t,y_t)$, by It$\h{o}$'s formula,
\begin{equation*}
  d\phi(\th_t, y_t) = \l(y_t\cdot\nb_\th\phi + (-\nb f(\th_t) - \b y_t)\cdot \nb_y\phi + \frac{1}{2}\a^2 \nb_\th\cdot\nb_\th\phi\r)dt + \a \nb_\th\phi dB_t.
\end{equation*}
Taking expectation of this equation and integrating over $[t, t+h]$ leads to
\begin{equation*}
  \frac{1}{h}\E\l(\phi(\th_{t+h}, y_{t+h}) - \phi(\th_t, y_t)\r) = \frac{1}{h}\int_t^{t+h} \E \l(y_s\cdot\nb_\th\phi + (-\nb f(\th_s) - \b y_s) \cdot\nb_y\phi + \frac{\a^2}{2} \nb_\th\cdot\nb_\th\phi\r) ds ,
\end{equation*}
which further gives,
\begin{equation*}
  \begin{aligned}
    &\frac{1}{h}\int \phi(\th, y) \l( \p(t+h,\th,y) - \p(t,\th,y)\r)\,d\th\,dy \\
    =& \frac{1}{h}\int_t^{t+h} \int \l(y\cdot\nb_\th\phi + (-\nb f(\th) - \b y) \cdot\nb_y\phi + \frac{\a^2}{2} \nb_\th\cdot\nb_\th\phi \r) \p(t,\th,y)\,d\th\,dy\, ds.
  \end{aligned}
\end{equation*}
Integrating by parts and letting $h\to0$ results in
\begin{equation*}
  \int \phi(\th, y) \pt_t\p\,d\th\,dy = \int \phi\l(-y\cdot\nb_\th\p +\nb f(\th)\cdot\nb_y\p +\nb_y\cdot (\b y\p) + \frac{\a^2}{2}\nb_\th\cdot\nb_\th\p \r) \,d\th\,dy,
\end{equation*}
which is true for any test function. Therefore, the PDF $\p(t,\th,y)$ satisfies
\begin{equation}
  \label{eq: origin pde}
  \pt_t\p + y\cdot\nb_\th\p -\nb f(\th)\cdot\nb_y\p =  \nb_y\cdot(\b y\p) + \frac{\a^2}{2} \nb_\th\cdot\nb_\th\p.
\end{equation}
In what follows, we consider the case where $\nb f$ is a perturbed linear function,
\begin{equation*}
  \nb f(\theta) = \w^2 \th + \v(\th).
\end{equation*}
A further change-of-variable of 
\begin{equation*}
  x = y, \qd v = -\w^2\th - \b y,
\end{equation*}
turns (\ref{eq: origin pde}) into
\begin{equation}
  \label{eq: VFP g}
  \pt_tg + v\cdot\nb_xg -\w^2 x\cdot\nb_vg = \b \nb_v\cdot\l( vg + \frac{1}{\beta}\nb_v g\r) +
  \v\(-\frac{1}{\w^2}(v +\b x)\)\cdot (\nb_xg-\b \nb_v g)
\end{equation}
with $g(t,x,v) = \p\(t,-\frac{1}{\w^2}(v +\b x), x\)$, $\beta = \frac{2\g}{\a^2\w^4}$.

\section{Main results and proof sketch}
\label{sec: main results}

When $\v = 0$, the steady state $M(x,v)$ of (\ref{eq: VFP g}),
\begin{equation}
\label{def of M}
M(x,v) = \Mx\Mv : = \l(\frac{1}{Z_1}e^{-\frac{\beta\w^2}{2}\lv x\rv^2}\r)\l(\frac{1}{Z_2}e^{-\frac{\beta}{2}\lv v\rv^2}\r),
\end{equation}
where $Z_1, Z_2$ are the normalization constants such that $\int \Mv dv = \int \Mx dx = 1$. However,
for general $f(\th)$, unfortunately there is no explicit form of the steady state. By denoting
$F(x,v)$ as the steady state of (\ref{eq: VFP g}), the weighted fluctuation function
\begin{equation*}
  h(t,x,v) = \frac{1}{M}\l[g(t,x,v) - F(x,v)\r]
\end{equation*}
satisfies the following equation,
\begin{equation}
\label{eq: VFP}
    \pt_th + \T h  = \L h +  \K h,
\end{equation}
where 
\begin{equation}
\label{eq: def of operators}
\begin{aligned}
  &\T = v\cdot\nb_x - \w^2x\cdot\nb_v \text{ is the transport operator};\\
  &\L = \b \l[-v\cdot\nb_v + \frac1\beta \nb_v\cdot\nb_v\r] =  \frac{\g}{\beta}\frac{1}{M}\nb_v\cdot(M\nb_v)  \text{ is the Linearized Fokker Planck operator};\\
  &\K h = \v\cdot\l(\nb_xh -\b\nb_vh\r) -\beta\v \cdot \(\w^2x-\g v\) h \text{ is the perturbation terms}.
\end{aligned}
\end{equation}
The above equation (\ref{eq: VFP}) is typically called the {\it microscopic equation} in the
literature. It is also convenient for the forthcoming analysis to define the inner product
$\la\cdot, \cdot \ra$ and the norm $\ll \cdot \rl_*$ as
\begin{equation}
  \begin{aligned}
    \label{def: norm star}
    &\la h,g \ra_\st = \int hg M \,dxdv, \qd\ll h \rl^2_* = \la h, h \ra_\st.
  \end{aligned}
\end{equation}
In addition, $\ll \cdot \rl^2$ is the standard $L^2$ norm with respect to the Lebesgue measure.

Under the above Gaussian measure $M$, the following Poincare inequality holds
\begin{equation}
\label{poincare ineq}
\begin{aligned}
  \ll h \rl^2_\st \leq \frac{1}{d\beta\min\{\w^2, 1\}}\(\ll \nb_xh\rl^2_\st + \ll\nb_vh\rl_\st^2\) ,\qd \text{for }\forall h\ s.t.\ \int hM dxdv= 0
\end{aligned}
\end{equation}
The following key assumption ensures various bounds of the perturbation $\v(\th)$.
\begin{assumption}
  \label{as: k}
  There exists a small constant $\cc>0$, such that,
  \begin{equation*}
    \begin{aligned}
      &\max_{i}\ll \v_i \rl_{L^\infty_{x,v}}, \ll \v\cdot x\rl_{L^\infty_{x,v}}, \ll \v\cdot
      v\rl_{L^\infty_{x,v}}, d \max_{i} \ll \v_i' \rl_{L^\infty_{x,v}}, \sum_{i} \ll \v_i'
      \rl_{L^\infty_{x,v}}, \ll \v'\cdot x\rl_{L^\infty_{x,v}}, \ll \v'\cdot v\rl_{L^\infty_{x,v}}\leq
      \cc
    \end{aligned}
\end{equation*}
\end{assumption}
The following theorem states an exponential decay bound for the fluctuation $h$.
\begin{theorem}
  \label{thm: main thm}
  Under Assumption \ref{as: k} with $\cc$ small enough, the fluctuation $h$ decay exponentially as
  follows,
  \begin{equation*}
    \ll h(t) \rl^2_\st \lesssim e^{-2(\mu - \e) t} H(0),
  \end{equation*} 
  where $H(0) = \ll \nb_xh(0) \rl^2_\st + C\ll \nb_v h(0) \rl^2_\st+ 2\h{C}\la \nb_xh(0),
  \nb_vh(0)\ra_\st$, $\e = \cc C_1$ for a constant $C_1$ depending on $\w, \g$ and
  \begin{equation}
    \label{eq: condition on g}
    \l\{ \begin{aligned}
      &\text{when }\g < 2\w: \qd\mu = \g, \qd C = \w^2,\qd\h{C} = \g/2;\\
      &\text{when }\g > 2\w:\qd\mu = \g - \sqrt{\g^2 - 4\w^2}, \qd C = \g^2/2-\w^2, \qd\h{C} = \g/2;\\
      &\text{when }\g = 2\w:\qd\text{for }\forall \d>0,\text{ there exists }C(\d), \h{C}(\d),\text{ such that the decay rate }\mu = \g - \d .\\
    \end{aligned}
    \r.
  \end{equation} 
  More specifically $C_1 =\(11 + 11C + 15\h{C}\)\cc \cdot \bb^2\cdot\frac{\max\{1,C\}}{C-\h{C}^2}$,
  where $\bb=\frac{\max\{1,\b,\g^2,\beta\g,\beta\w^2\}}{\min\{1,\w^2\}}$.
\end{theorem}

\begin{remark}
\label{rmk: staleness conv}
{ {\bf How the learning rate and staleness affect the convergence rate?} }When the perturbation $\cc$
  is small, the decay rate is dominated by $e^{-2\mu t}$.  By the
definition of $\d_t$, one has $t = k\d_t = k\sqrt{\eta(1-\k)}$ with $k$ the number of steps, $\eta$ the learning rate and $\k$ the staleness rate. Inserting the definition of $\g = \sqrt{((1-\k)/\eta)}$ into  (\ref{eq: condition on g}), the dominated decay rate  $e^{-2\mu t}$ can also be written as,
  \begin{equation}
  \label{eq: condition on rate}
  \l\{ \begin{aligned}
    &\text{when }\eta > \frac{1}{4\w^2}(1-\k): \qd\mu t = (1-\k)k;\\
    &\text{when }\eta  < \frac{1}{4\w^2}(1-\k):\qd\mu t = (1-\k)k - \(\sqrt{(1-\k)^2 - 4\w^2(1-\k)\eta}\)k.
  \end{aligned}
  \r.
  \end{equation} 
  From the above discussion, we make two observations:
  \begin{itemize}
  \item [--] The learning rate should not be smaller than $\frac{1}{4\w^2}(1-\k)$. For a fixed staleness rate $\k$, when the learning rate is larger than
    the threshold $\frac{1}{4\w^2}(1-\k)$, the convergence rate is a constant only depending on 
    $(1-\k)$. While the learning rate is smaller than this threshold, the convergence rate will become
    slower as the learning rate becomes smaller.
  \item [--] Longer delays result in slower convergence rate.
  For a fixed learning rate, the optimal
    decay rate $e^{-2(1-\k)k}$ only relates to the staleness of the system. If the system has more
    delayed readings from the local workers, i.e., $(1-\k)$ is smaller, then the convergence rate is
    slower.
  \end{itemize}
  All the above discussion is based on the assumption that  $\eta$ is small enough so that the SME-ASGD is a good approximation for ASGD. In other words, we assume $\w$ is large here, hence the threshold $\frac{1}{4\w^2}(1-\k)$ is still in the valid regime. 
\end{remark}

\begin{remark} 
\label{rmk: ASGD SGD}  {{\bf When is ASGD more efficient than SGD?}}
  Assume we have $m$ local workers and the learning rate is larger than the threshold $\eta >
  \frac{1}{4\w^2}(1-\k)$. When the perturbation $\cc$ is small, for single batch SGD, the decay rate
  is $e^{-2 k}$ after $k$ steps, while for ASGD, the decay rate is $e^{-2(1-\k)k}$ after calculating
  $k$ gradients. Since now we have $m$ local workers, for the same amount of time, the decay rate
  for ASGD becomes $e^{-2(1-\k)m k}$. Therefore, as long as $(1-\k)m > 1$, ASGD will be more
  efficient than SGD. Since the expectation of the random staleness $\tau_k$ is $\frac{1}{1-\k}$, in
  other words, when the number of local workers is larger than the expected 
  staleness, then ASGD is more efficient than SGD. 
\end{remark}


The proof of the theorem is given in Section \ref{sec: proof of main thm}. The main ingredient of the
proof is the following Lyapunov functional $H(t)$,
\begin{equation}
  \label{def: H}
  H(t) =\ll \nb_xh \rl^2_\st + C\ll \nb_vh \rl^2_\st + 2\h{C}\la\nb_xh, \nb_vh \ra_\st
\end{equation}
where $C, \h{C}$ are constants to be determined later.  Note that,
\begin{equation}
  \label{eq: pt_tH}
  \frac{d}{dt}H(t) =\frac{d}{dt}\(\ll \nb_xh \rl^2_\st + C\ll \nb_vh \rl^2_\st\) +
  2\h{C}\frac{d}{dt}\la\nb_xh, \nb_vh \ra_\st.
\end{equation}
The first two terms can be estimated with an energy estimation given in Lemma \ref{lemma: nb_xh and
  nb_vh} of $\nb_x(\ref{eq: VFP})$, $\nb_v(\ref{eq: VFP})$, while the estimation of the last term is
given in Lemma \ref{lemma: nb_xh nb_vh}. Combining the result in Lemmas \ref{lemma: nb_xh and nb_vh}
and \ref{lemma: nb_xh nb_vh} leads to,
\begin{equation}
  \frac{1}{2}\pt_t H(t) + \t{C} H(t) \leq 0.
\end{equation}
Finally, the exponential decay of $\ll h(t) \rl^2_\st$ can be derived from this inequality and the
relationship between $H(t)$ and $\ll h(t) \rl^2$.

\section{Proof of Theorem \ref{thm: main thm} }
\label{sec: proof of main thm}

The following proposition summarizes a few equalities and inequalities, which will be used
frequently in the proof of the main theorem. The proofs are provided in Appendix \ref{apdx:
  inequlities}.
\begin{proposition}
  \label{prop: inequlities}
  For $\forall g(t,x,v), h(t,x,v)$, the following statements hold
  \begin{itemize}
  \item [(a)] $\ds \la \T g,  h\ra_\st = -\la Th, g\ra_\st, \qd \la \T h,  h\ra_\st = 0.$
  \item [(b)] $\ds \la Lg, h\ra_\st = -\frac\g\beta\la \nb_v g, \nb_vh \ra_\st. $
  \item [(c)] $\ds \la \K g, g\ra_\st \leq \cc\bb \ll g\rl^2_\st, \qd \ds \la \K g, h\ra_\st + \la \K h, g\ra_\st \leq \cc\bb\( \ll g\rl^2_\st +  \ll h\rl^2_\st\)$, \qd where $\bb = \frac{\max\{1,\b,\g^2,\beta\g, \beta\w^2\}}{\min\{1,\w^2\}}$.
  \item [(d)] $\ds \la \nb_x(\K h), \nb_xh \ra_\st \leq \frac{11}{2}\cc\bb^2 \ll \nb_xh\rl^2_\st +2\cc\bb^2\ll \nb_vh \rl^2_\st.$
  \item [(e)] $\ds \la \nb_v(\K h), \nb_vh \ra_\st \leq \frac{11}2\cc\bb^2\ll \nb_vh\rl^2_\st + 2\cc\bb^2\ll\nb_xh\rl^2_\st $.
  \item [(f)] $\ds \la \nb_x(\K h), \nb_vh \ra_\st + \la \nb_v(\K h), \nb_xh \ra_\st\leq \frac{15}2\cc\bb^2\ll\nb_xh\rl^2_\st +\frac{15}2\cc\bb^2\ll\nb_vh\rl^2_\st$.
  \end{itemize}
\end{proposition}

The following lemma is the energy estimation of $\nb_x$(\ref{eq: VFP}) and $\nb_v$(\ref{eq: VFP}).
\begin{lemma}
  \label{lemma: nb_xh and nb_vh}
The weighted fluctuation function $h(t,x,v)$ satisfies
\begin{equation*}
  \begin{aligned}
    &\frac{d}{dt}\(\ll \nb_xh \rl^2_\st+C\ll \nb_xh \rl^2_\st \) + 2\frac\g\beta\sum_{i} \int M (\lv \pt_{v_i}\nb_x h \rv^2  + C\lv \pt_{v_i} \nb_vh \rv^2 )dv \\
    \leq & -2( C - \w^2) \la \nb_xh, \nb_vh\ra_\st - 2\g C \ll \nb_v h \rl^2_\st  + \({11} + 4C\)\cc\bb \ll \nb_x h \rl^2_\st +  \(4+11C\)\cc\bb \ll \nb_v h \rl^2_\st.
  \end{aligned}
\end{equation*}
\end{lemma}

\begin{proof}
  After taking $\nb_x$ and $\nb_v$ to (\ref{eq: VFP}), multiplying them with $\nb_xhM$ and
  $\nb_vhM$ respectively, and integrating over $dxdv$, one obtains
  \begin{equation*}
    \begin{aligned}
      &\frac12\frac{d}{dt}\ll \nb_xh \rl^2_\st + \la T\nb_xh, \nb_xh \ra_\st -\w^2 \la \nb_vh, \nb_xh\ra_\st = \la L\nb_xh, \nb_xh \ra_\st + \la \nb_x(\K h), \nb_xh \ra_\st .\\
      &\frac12\frac{d}{dt}\ll \nb_vh \rl^2_\st + \la T\nb_vh, \nb_vh \ra_\st + \la \nb_xh, \nb_vh\ra_\st = \la L\nb_vh, \nb_vh \ra_\st  - \b\la \nb_vh, \nb_vh \ra_\st  + \la \nb_v(\K h), \nb_vh \ra_\st.
    \end{aligned}
  \end{equation*}
  By invoking Proposition \ref{prop: inequlities}/(a), the second term of the LHS of each equation
  vanishes. Multiplying these two equations with $1$ and $C$ respectively and adding them
  together gives rise to 
  \begin{equation}
    \label{eq: lemma44_1}
    \begin{aligned}
      &\frac12\frac{d}{dt}\(\ll \nb_xh \rl^2_\st+C\ll \nb_xh \rl^2_\st \) - \(\la L\nb_xh, \nb_xh \ra_\st + C\la L\nb_vh, \nb_vh \ra_\st\) \\
      \leq & -( C - \w^2) \la \nb_xh, \nb_vh\ra_\st - \g C \ll \nb_v h \rl^2_\st + \la \nb_x(\K h), \nb_xh \ra_\st + C\la \nb_v(\K h), \nb_vh \ra_\st
    \end{aligned}
  \end{equation}
  After applying Proposition \ref{prop: inequlities}/(b),(d),(e), one obtains
  \begin{equation}
    \label{eq: lemma44_1}
    \begin{aligned}
      &\frac12\frac{d}{dt}\(\ll \nb_xh \rl^2_\st+C\ll \nb_xh \rl^2_\st \) + \frac\g\beta\(\sum_{i} \ll\pt_{v_i}\nb_x h \rl_\st^2  + C\ll \pt_{v_i} \nb_vh \rl_\st^2\) \\
      \leq & -( C - \w^2) \la \nb_xh, \nb_vh\ra_\st - \g C \ll \nb_v h \rl^2_\st + \(\frac{11}2 + 2C\)\cc\bb \ll \nb_x h \rl^2_\st +  \(2+\frac{11}2C\)\cc\bb \ll \nb_v h \rl^2_\st.
    \end{aligned}
  \end{equation}
\end{proof}

\begin{lemma}
  \label{lemma: nb_xh nb_vh}
  \begin{equation*}
    \begin{aligned}
      &\frac{d}{dt}\la \nb_xh,  \nb_vh \ra_\st + 2\frac{\g}{\beta}\sum_{i} \la  \pt_{v_i}\nb_x h, \pt_{v_i} \nb_vh \ra_\st \\
      \leq & -\g \la \nb_xh, \nb_vh\ra_\st + \w^2 \ll \nb_v h \rl^2_\st - \ll \nb_xh \rl^2_\st+\frac{15}2\cc\bb^2\ll\nb_xh\rl^2_\st +\frac{15}2\cc\bb^2\ll\nb_vh\rl^2_\st.
    \end{aligned}
  \end{equation*}
\end{lemma}

\begin{proof}
  Taking $\nb_x$ and $\nb_v$ to (\ref{eq: VFP}), multiplying them by $\nb_vh M$ and $\nb_xh M$
  respectively, and integrating over $dxdv$, one obtains
  \begin{equation*}
    \begin{aligned}
      &\la \frac{d}{dt} \nb_xh, \nb_vh\ra_\st  + \la T\nb_xh, \nb_vh \ra_\st -\w^2 \la \nb_vh, \nb_vh\ra_\st = \la L\nb_xh, \nb_vh \ra_\st + \la \nb_x(\K h), \nb_vh \ra_\st .\\
      &\la \frac{d}{dt} \nb_vh, \nb_xh\ra_\st + \la T\nb_vh, \nb_xh \ra_\st + \la \nb_xh, \nb_xh\ra_\st = \la L\nb_vh, \nb_xh \ra_\st  - \b\la \nb_vh, \nb_xh \ra_\st  + \la \nb_v(\K h), \nb_xh \ra_\st.
    \end{aligned}
  \end{equation*}
  From Proposition \ref{prop: inequlities}/(a), the sum of the second terms on the LHS of both
  equations vanishes. Applying Proposition \ref{prop: inequlities}/(b) to the first term on the RHS
  of both equations combine them into a single term. Finally, summing the above two equations and
  applying Proposition \ref{prop: inequlities}/(f) to the last terms leads to
  \begin{equation}
    \begin{aligned}
      &\frac{d}{dt}\la \nb_xh,  \nb_vh \ra_\st + 2\frac\g\beta\sum_{i} \la  \pt_{v_i}\nb_x h, \pt_{v_i} \nb_vh \ra_\st \\
      \leq & -\g \la \nb_xh, \nb_vh\ra_\st + \w^2 \ll \nb_v h \rl^2_\st - \ll \nb_xh \rl^2_\st +\frac{15}2\cc\bb^2\ll\nb_xh\rl^2_\st +\frac{15}2\cc\bb^2\ll\nb_vh\rl^2_\st.
    \end{aligned}
  \end{equation}
\end{proof}

\paragraph{The proof of Theorem \ref{thm: main thm} }

By combining the results in Lemma \ref{lemma: nb_xh and nb_vh} and Lemma \ref{lemma: nb_xh nb_vh},
one concludes that
\begin{equation}
  \label{eq: Ht 1}
  \begin{aligned}
    &\frac{d}{dt}H(t) +\int [ \nb_xh, \nb_vh] K [ \nb_xh, \nb_vh]^\top\frac1Mdxdv \\
    & + \int \frac{2\g}{\beta}\sum_{i} [\pt_{v_i}\nb_x h,  \pt_{v_i} \nb_vh ]P [\pt_{v_i}\nb_x h,  \pt_{v_i} \nb_vh ]^\top\frac1Mdxdv \\
    \leq &\(11 + 4C + 15\h{C}\)\cc\bb^2\ll\nb_xh\rl^2_\st +\(4 + 11C + 15\h{C}\)\cc\bb^2\ll\nb_vh\rl^2_\st,
  \end{aligned}
\end{equation}
where
\begin{equation}
  \label{def of K}
  K = \left[ 
    \begin{aligned} 
      & 2\h{C}I_d  &(C - \w^2+\g\h{C})I_d\\
      &(C - \w^2+\g\h{C})I_d &(2\g C - 2\w^2\h{C})I_d
    \end{aligned}\right],\qd
  P =  \left[ 
    \begin{aligned} 
    & I_d  &\h{C}I_d\\
    &\h{C}I_d &CI_d
    \end{aligned}\right].
\end{equation}
Note that $K$ can be decomposed as,
\begin{equation}
\label{def of Q, P}
    K = QP + P Q^\top , \qd 
    \text{wtih } Q =  \left[ 
    \begin{aligned} 
    & 0I_d  &I_d\\
    &-\w^2I_d &\g I_d
    \end{aligned}\right].
\end{equation}
By invoking Lemma 4.3 in \cite{arnold2014sharp}, we know that there exists a positive definite matrix
$P$ such that,
\begin{equation}
  \label{eq: relation K P}
  K = QP + Q^\top P \geq 2 \mu P, \qd \text{with } \mu = \min\{\mathrm{Re}(\lam): \lam \text{ is an eigenvalue of }Q\}.
\end{equation}
The value of $\mu$, $C$, and $\h{C}$ can be separated into three cases. 
\begin{itemize}
\item [-]  case 1: $\ds \g < 2\w: \qd\mu = \g, \qd C = \w^2,\qd\h{C} = \g/2$ 
\item [-]  case 2: $\ds \g > 2\w: \qd\mu = \g - \sqrt{\g^2 - 4\w^2}, \qd C = \g^2/2-\w^2, \qd\h{C} = \g/2$
\item [-] case 3:  $\ds \g = 2\w: \qd$ For $\forall \d>0$, there exists $\mu = \g -\d, C(\d),\h{C}(\d)$, such that  (\ref{eq: relation K P}) holds. 
\end{itemize}
Inserting $K\geq2\mu P$ and using the fact that
$H(t)=\int[\nb_xh,\nb_vh]P[\nb_xh,\nb_vh]^\top\frac1Mdxdv$, we can bound the second term in
(\ref{eq: Ht 1}) from below by $2\mu H(t)$. By the positive definiteness of $P$, the third term is
always positive. Therefore,
\begin{equation*}
  \begin{aligned}
    &\frac{d}{dt}H(t) + 2\mu H(t)\leq \(11 + 4C + 15\h{C}\)\cc\bb^2\ll\nb_xh\rl^2_\st +\(4 + 11C + 15\h{C}\)\cc\bb^2\ll\nb_vh\rl^2_\st,\\
    &\frac{d}{dt}H(t) + 2(\mu -\e) H(t) \leq -2\e H(t) + \(11 + 11C + 15\h{C}\)\cc\bb^2\(\ll\nb_xh\rl^2_\st +\ll\nb_vh\rl^2_\st\),\\
    &\frac{d}{dt}H(t) + 2(\mu -\e) H(t) \leq -\e\ll \nb_xh +\h{C}\nb_vh\rl^2_\st-\e\ll \sqrt{C}\nb_vh + \frac{\h{C}}{\sqrt{C}}\nb_xh\rl^2_\st-\e\(C - \h{C}^2 \)\ll \nb_v h \rl^2_\st\\
    & - \frac{\e}{C}\(C - \h{C}^2 \)\ll \nb_x h \rl^2_\st  + \(11 + 11C + 15\h{C}\)\cc\bb^2\(\ll\nb_xh\rl^2_\st +\ll\nb_vh\rl^2_\st\).
  \end{aligned}
\end{equation*}
The RHS is less than $0$ for all $\nb_xh, \nb_vh$ if 
\begin{equation*}
  \begin{aligned}
    & \(11 + 11C + 15\h{C}\)\cc\bb^2 \leq \e(C - \h{C}^2)\min\l\{1,\frac{1}{C}\r\},
  \end{aligned}
\end{equation*}
which implies that as long as $\cc$ is sufficiently small one has
\begin{equation}
  \label{eq: Ht 2}
  \begin{aligned}
    &\frac{d}{dt}H(t) + 2(\mu -\e) H(t) \leq 0
  \end{aligned}
\end{equation}
for $\e = \(11 + 11C + 15\h{C}\)\cc\bb^2\frac{\max\{1,C\}}{C - \h{C}^2}$. By integrating (\ref{eq:
  Ht 2}) over time and applying Grownwall's inequality to it, we obtain
\begin{equation*}
  \begin{aligned}
    &H(t) \leq H(0)e^{-2(\mu - \e)t}.
  \end{aligned}
\end{equation*}
By the Poincare inequality (\ref{poincare ineq}) and the positive definiteness of $P$, one can bound
$H(t)$ from below by $\ll h\rl^2_\st$ up to a constant,
\begin{equation*}
  \begin{aligned}
    &H(t) \gtrsim (\ll \nb_x h \rl^2_\st + \ll \nb_v h \rl^2_\st ) \gtrsim \ll h \rl^2_\st.
  \end{aligned}
\end{equation*}
Therefore, we can conclude
\begin{equation*}
  \begin{aligned}
    &\ll h(t) \rl^2_\st \lesssim H(0)e^{-2(\mu - \e)t}.
  \end{aligned}
\end{equation*}



\appendix
\section*{Appendices}
\addcontentsline{toc}{section}{Appendices}
\renewcommand{\thesubsection}{\Alph{subsection}}

\subsection{The proof of Proposition \ref{prop: inequlities}}
\label{apdx: inequlities}
\setcounter{equation}{0}
\setcounter{theorem}{0}
\renewcommand\theequation{A.\arabic{equation}}
\begin{proof}

\begin{itemize}
\item [(a)] 
\end{itemize}
The first equation can be proved via integration by parts,
\begin{equation*}
  \begin{aligned}
    &\la Tg, h\ra_\st = \int \(v\cdot \nb_xg  -\w^2x\cdot\nb_vg \)h M dxdv \\
	=& \int \(-v\cdot \nb_xh + \w^2x\cdot\nb_vh \)g + gh(-v\cdot \nb_xM + \w^2x\cdot\nb_vM) dxdv\\
	=&-\la Th, g\ra_\st + \int gh\(v\cdot(\beta\w^2x) - \w^2x\cdot(\beta v)\) dxdv = -\la Th, g\ra_\st.
\end{aligned}
\end{equation*}
The second equation is directly followed by $\ds \la Th, h\ra_\st = -\la Th, h\ra_\st$.

\begin{itemize}
\item [(b)] 
\end{itemize}
This equation can be obtained also by integration by parts,
\begin{equation*}
  \begin{aligned}
	&\la Lg, h\ra_\st = \frac{\g}{\beta}\int \frac{1}{M}\nb_v\cdot(M\nb_vg) hMdxdv =
    -\frac{\g}{\beta}\int M\nb_vg \cdot\nb_vh dxdv = -\frac{\g}{\beta}\la\nb_vg , \nb_vh \ra_\st .
  \end{aligned}
\end{equation*}

\begin{itemize}
\item [(c)] 
\end{itemize}
The first equation can be written as
\begin{equation}
  \label{eq: c1}
  \begin{aligned}
    &\la \K g, g\ra_\st = \la \v\cdot\l(\nb_xg -\b\nb_vg\r), g\ra_\st - \la \beta\v \cdot \(\w^2x-\g v\) g, g\ra_\st\\
    =&\frac{1}{2} \la \v, (\nb_x -\b\nb_v)g^2  \ra_\st  - \la \beta\v \cdot \(\w^2x-\g v\) g, g\ra_\st.
  \end{aligned}
\end{equation}
By integrating by parts the first term and using the definition of $\v = \v\(-\frac{1}{\w^2}(v +\b
x)\), M = \text{exp}(-\beta\w^2|x|^2/2 + \beta|v|^2/2)$, one has
\begin{equation}
  \label{eq: c2}
  \begin{aligned}
    &\frac{1}{2} \la \v, (\nb_x -\b\nb_v)g^2  \ra_\st =-\frac{1}{2} \la (\nb_x -\b\nb_v)\cdot\v, g^2  \ra_\st - \frac{1}{2} \la \v\cdot(\nb_x -\b\nb_v)M, g^2  \ra\\
    =& -\frac{1}{2} \la -\frac{\g}{\w^2}\nb\cdot\v-\g(-\frac{1}{\w^2})\nb\cdot\v, g^2  \ra_\st + \frac{1}{2}\la \beta\v \cdot \(\w^2x-\g v\), g^2\ra_\st = \frac{1}{2}\la \beta\v \cdot \(\w^2x-\g v\), g^2\ra_\st.
  \end{aligned}
\end{equation}
Inserting the above equation into (\ref{eq: c1}) gives rise to
\begin{equation*}
  \begin{aligned}
    &\la \K g, g\ra_\st =  -\frac12\la \beta\v \cdot \(\w^2x-\g v\), g^2\ra_\st \leq \frac\cc2(\beta\w^2 + \beta\g)\ll g \rl^2_\st \leq \cc\bb \ll g \rl^2_\st,
  \end{aligned}
\end{equation*}
where we use the assumption $\ll \e\cdot x \rl_{L^\infty}, \ll \e\cdot v \rl_{L^\infty} \leq \cc$ in
Assumption \ref{as: k} at the first inequality.

Now we estimate $\la \K g, h\ra_\st + \la \K h, g\ra_\st $. First, similar to (\ref{eq: c2}), by the
definition of $\v, M$, one has
\begin{equation*}
  \begin{aligned}
    &\la \K g, h\ra_\st = \la \v\cdot\l(\nb_xg -\b\nb_vg\r), h\ra_\st - \la \beta\v \cdot \(\w^2x-\g v\) g, h\ra_\st\\
    =&-\la  (\nb_x -\b\nb_v)\v g, h  \ra_\st-\la  \v g,  (\nb_x -\b\nb_v)h  \ra_\st -\la  \v  (\nb_x -\b\nb_v)M g, h  \ra \\
    &- \la \beta\v \cdot \(\w^2x-\g v\) g, h\ra_\st\\
    =& 0-\la  \v g,  (\nb_x -\b\nb_v)h\ra_\st + 0.
\end{aligned}
\end{equation*}
Hence, 
\begin{equation*}
  \begin{aligned}
    \la \K g, h\ra_\st+ \la \K h, g\ra_\st  =&-\la  \v (\nb_x -\b\nb_v)h, g\ra_\st + \la \v\cdot\l(\nb_xh -\b\nb_vh\r), g\ra_\st - \la \beta\v \cdot \(\w^2x-\g v\) h, g\ra_\st\\
    =&- \la \beta\v \cdot \(\w^2x-\g v\) h, g\ra_\st \leq \frac\cc2\(\beta\w^2+\beta\g\)\(\ll h\rl^2_\st+\ll g\rl^2_\st\)\\
    \leq &\cc\bb\(\ll h\rl^2_\st+\ll g\rl^2_\st\).
  \end{aligned}
\end{equation*}

\begin{itemize}
\item [(d)] 
\end{itemize}
By the definition of $\K$ in (\ref{eq: def of operators}),
\begin{equation}
  \label{eq: d1}
  \begin{aligned}
    \la \nb_x(\K h), \nb_xh \ra_\st =& \la \K\nb_xh, \nb_xh \ra_\st + \la \nb_x\v(\nb_xh -  \b\nb_vh), \nb_xh \ra_\st  \\
    &- \beta\la \nb_x\v(\w^2x-\g v)h, \nb_xh \ra_\st - \beta\w^2\la \v h, \nb_xh \ra_\st\\
    \leq&\cc \bb\ll \nb_xh\rl^2_\st + \la \nb_x\v(\nb_xh -  \b\nb_vh), \nb_xh \ra_\st \\
    &- \beta\la \nb_x\v(\w^2x-\g v)h, \nb_xh \ra_\st+\frac12 \beta\w^2\cc(d\ll h \rl^2_\st + \ll \nb_xh \rl^2_\st),
\end{aligned}
\end{equation}
where we apply the inequality $(c)$ to the first term and Assumption \ref{as: k} $\max_{i}\ll \v_i
\rl_{L^\infty} \leq \cc$ to the last term of the above inequality. We will then estimate the second
and third terms.
\begin{equation}
  \label{eq: d2}
\begin{aligned}
  &\la \nb_x\v(\nb_xh -  \b\nb_vh), \nb_xh \ra_\st  = \sum_{i,j}\la \pt_{x_j}\v_i(\pt_{x_i}h -  \b\pt_{v_i}h), \pt_{x_j}h \ra_\st \\
  =& -\frac{\g}{\w^2}\sum_{i,j}\la \v_i'(\pt_{x_i}h -  \b\pt_{v_i}h), \pt_{x_j}h \ra_\st \leq \frac{\g}{\w^2} \sum_{i,j} \ll \v_i' \rl_{L^\infty} \(\frac12\ll\pt_{x_i}h\rl^2_\st + \frac\b2\ll \pt_{v_i}h\rl^2_\st + \ll \pt_{x_j}h\rl^2_\st\) \\
  \leq & \frac{\g}{\w^2}\( d \max_{i} \ll \v_i' \rl_{L^\infty} \(\frac12\ll \nb_xh\rl^2_\st + \frac\g2\ll \nb_vh \rl^2_\st\) +  \(\sum_{i} \ll \v_i' \rl_{L^\infty}\) \ll \nb_xh\rl^2_\st\) \\
  \leq& \frac{\g}{\w^2}\cc\( \frac12\ll \nb_xh\rl^2_\st + \frac\g2\ll \nb_vh \rl^2_\st + \ll \nb_xh\rl^2_\st\) \leq \cc\bb\(\frac{3}{2}\ll \nb_xh\rl^2_\st + \frac12\ll \nb_vh \rl^2_\st \)
\end{aligned}
\end{equation}
where $ d \max_{i} \ll \v_i' \rl_{L^\infty}, \sum_{i} \ll \v_i' \rl_{L^\infty}\ \leq \cc, $ are used
in the second inequality from last. The third term in (\ref{eq: d1}) can be bounded by,
\begin{equation}
  \label{eq: d3}
  \begin{aligned}
    &- \beta\la \nb_x\v(\w^2x-\g v)h, \nb_xh \ra_\st = -\beta\sum_{i,j}\la \pt_{x_j}\v_i(\w^2x_i-\g v_i)h, \pt_{x_j}h \ra_\st \\
    =& -\frac{\g\beta}{\w^2}\sum_{j}\la \sum_i\v_i'(\w^2x_i - \g v_i) h, \pt_{x_j}h \ra_\st \leq \frac{\g\beta}{\w^2} \cc \(\frac{d}{2}\w^2\ll h \rl^2_\st + \frac{d}{2}\g\ll h \rl^2_\st + \ll \nb_x h \rl^2_\st\)\\
    \leq &\cc\bb  \(d\beta \ll h \rl^2_\st + \ll \nb_xh \rl^2_\st\),
  \end{aligned}
\end{equation}
where $\ll \v'\cdot x\rl_{L^\infty}, \ll \v'\cdot v\rl_{L^\infty} \leq \cc$ are used in last
inequality. Then inserting (\ref{eq: d2}), (\ref{eq: d3}) into (\ref{eq: d1}) leads to
\begin{equation}
  \label{eq: d4}
  \begin{aligned}
    \la \nb_x(\K h), \nb_xh \ra_\st \leq& 4\cc\bb \ll \nb_xh\rl^2_\st +\frac12\cc\bb\ll \nb_vh
    \rl^2_\st+\frac32\cc\bb\(d\beta\ll h \rl^2_\st\).
  \end{aligned}
\end{equation}
Now applying the Poincare inequality (\ref{poincare ineq}) under the Gaussian measure to (\ref{eq:
  d4}) gives rise to
\begin{equation*}
  \begin{aligned}
	\la \nb_x(\K h), \nb_xh \ra_\st \leq& \(4\cc\bb + \frac{3}{2}\cc\bb^2\) \ll \nb_xh\rl^2_\st +\(\frac12\cc\bb + \frac32\cc\bb^2\)\ll \nb_vh \rl^2_\st.
  \end{aligned}
\end{equation*}

\begin{itemize}
\item [(e)] 
\end{itemize} 
Similar to the proof of (d), one has
\begin{equation*}
  \begin{aligned}
    &\la \nb_v(\K h), \nb_vh \ra_\st \\
    =& \la \K\nb_vh, \nb_vh \ra_\st + \la \nb_v\v(\nb_xh -  \b\nb_vh), \nb_vh \ra_\st  \\
    &- \beta\la \nb_v\v(\w^2x-\g v)h, \nb_vh \ra_\st + \beta\g\la \v h, \nb_vh \ra_\st\\
    \leq& \cc\bb\ll \nb_vh\rl^2_\st + \cc\bb\( \frac12\ll\nb_xh\rl^2_\st +  \frac32\ll\nb_vh\rl^2_\st  \)\\
    &+\cc\bb\( d\beta\ll h\rl_\st^2  + \ll\nb_vh \rl^2_\st\)+ \cc\bb\(\frac12d\beta\ll h \rl^2_\st +\frac12 \ll \nb_vh \rl^2_\st\)\\
    \leq &\(4\cc\bb+\frac32\cc\bb^2\)\ll \nb_vh\rl^2_\st + \(\frac12\cc\bb+\frac32\cc\bb^2\)\ll\nb_xh\rl^2_\st .
\end{aligned}
\end{equation*}

\begin{itemize}
\item [(f)] 
\end{itemize}
Similar to the proof of (d), one has
\begin{equation*}
  \begin{aligned}
    &\la \nb_x(\K h), \nb_vh \ra_\st + \la \nb_v(\K h), \nb_xh \ra_\st \\
    =&\( \la \K\nb_xh, \nb_vh \ra_\st +\la \K\nb_vh, \nb_xh \ra_\st\)+ \la \nb_x\v(\nb_xh -  \b\nb_vh), \nb_vh \ra_\st  - \beta\la \nb_x\v(\w^2x-\g v)h, \nb_vh \ra_\st\\
    & - \beta\w^2\la \v h, \nb_vh \ra_\st + \la \nb_v\v(\nb_xh -  \b\nb_vh), \nb_xh \ra_\st- \beta\la \nb_v\v(\w^2x-\g v)h, \nb_xh \ra_\st + \beta\g\la \v h, \nb_xh \ra_\st\\
    \leq&\(\frac92\cc\bb + 3\cc\bb^2\)\ll\nb_xh\rl^2_\st + \(\frac92\cc\bb + 3\cc\bb^2\)\ll\nb_vh\rl^2_\st.
  \end{aligned}
\end{equation*}

\end{proof}


\bibliographystyle{plain}

\end{document}